\documentclass{amsart}
\usepackage{graphicx}
\usepackage{pslatex}
\usepackage{amsmath,amsfonts}
\usepackage{mathrsfs}
\usepackage{rotating}
\usepackage{amssymb}
\usepackage{verbatim}
\usepackage{color}
\usepackage{setspace}
\usepackage{hyperref}

\newtheorem{theorem}{Theorem}[section]
\newtheorem{lemma}[theorem]{Lemma}
\newtheorem{proposition}[theorem]{Proposition}
\newtheorem{corollary}[theorem]{Corollary}

\theoremstyle{definition}
\newtheorem{definition}[theorem]{Definition}

\theoremstyle{remark}
\newtheorem{remark}[theorem]{Remark}

\numberwithin{equation}{section}

\def\R{\mathbb R}
\def\C{\mathbb C}

\def\N{\mathbb N}
\def\S{\mathscr S}

\def\supp{\text{supp}}
\def\({\left(}
\def\){\right)}
\def\[{\left[}
\def\]{\right]}
\def\<{\left<}
\def\>{\right>}
\def\less{\lesssim}
\def\more{\gtrsim}

\begin{document}

\title{Multilinear Local Tb for Square Functions}

\author{Ana Grau de la Herr\'an}
\address{Department of Mathematics and Statistics\\ University of Helsinki\\ Helsinki  FI 00014}
\email{ana.grau@helsinki.fi}
\author{Jarod Hart}
\address{Department of Mathematics\\ University of Kansas\\ Lawrence  KS 66045}
\email{ jhart@math.ku.edu}
\author{Lucas Oliveira}
\address{Departamento de Matem\'{a}tica\\UFRGS\\ Porto Alegre RS 91509-900}
\email{lucas.oliveira@ufrgs.br}

\thanks{Grau de la Herr\'an was supported in part by NSF \#DMS 1101244 and ERC Starting Grant ``Analytic-probabilistic  
methods for borderline singular integrals".}
\thanks{Hart was supported in part by NSF Grant \#DMS1069015.}
\thanks{Oliveira was supported in part by CAPES-Processo 2314118}


\date{August 27, 2012.}

\dedicatory{ }


\begin{abstract}
In the present work we extend a local Tb theorem for square functions of Christ \cite{C} and Hofmann \cite{Ho} to the multilinear setting. We also present new $BMO$ type interpolation result for square functions associated to multilinear operators.  These square function bounds are applied to prove a multilinear local Tb theorem for singular integral operators.
\end{abstract}

\maketitle

\section{Introduction}

Consider the family of multilinear of operators $\{\Theta_{t}\}_{t>0}$ given by
\begin{align}
    \Theta_t(f_1,...,f_m)(x)=\int_{\R^{mn}}\theta_t(x,y_1,...,y_m)\prod_{i=1}^mf_i(y_i)dy_i\label{theta}
\end{align}
where $\theta_t:\R^{(m+1)n}\rightarrow\C$ and the square functions associated to $\{\Theta_t\}_{t>0}$
\begin{align}
S(f_1,...,f_m)(x)=\(\int_0^\infty|\Theta_t(f_1,...,f_m)(x)|^2\frac{dt}{t}\)^\frac{1}{2}\label{sqfunction}
\end{align}
where   $f_i$ for $i=1,...,m$ are initially functions in $ C_{0}^{\infty}(\R^{n})$ (smooth with compact support).  The purpose of this work is to find appropriate cancellation conditions on $\theta_t$ and indices $p,p_1,...,p_m$ that guarantee $L^p$ boundedness of the square functions $S$ of the form
\begin{align}
||S(f_1,...,f_m)||_{L^p}\less\prod_{i=1}^m||f||_{L^{p_i}}\label{Lpbound}
\end{align}
given that $\theta_t$ satisfies some size and regularity estimates.  In particular, we assume that $\theta_t$ satisfies for all $x,y_1,...,y_m,x', y_1',...,y_m'\in\R^n$
\begin{align}
&|\theta_t(x,y_1,...,y_m)|\less\frac{t^{-mn}}{\prod_{i=1}^m(1+t^{-1}|x-y_i|)^{N+\gamma}}\label{size}\\
&|\theta_t(x,y_1,...,y_m)-\theta_t(x,y_1,...,y_i',...,y_m)|\less\frac{t^{-mn}(t^{-1}|y_i-y_i'|)^\gamma}{\prod_{i=1}^m(1+t^{-1}|x-y_i|)^{N+\gamma}}\label{regy}\\
&|\theta_t(x,y_1,...,y_m)-\theta_t(x',y_1,...,y_m)|\less\frac{t^{-mn}(t^{-1}|x-x'|)^\gamma}{\prod_{i=1}^m(1+t^{-1}|x-y_i|)^{N+\gamma}}\label{regx}
\end{align}
for some $N>n$ and $0<\gamma\leq1$.   It follows from a scaling argument that if \eqref{Lpbound} holds, then the indices $p,p_1,...,p_m$ must satisfy the H\"older type relationship
\begin{align}
\frac{1}{p}=\sum_{i=1}^m\frac{1}{p_i}.\label{Holder}
\end{align}
So throughout this work we assume that all indices $p,p_1,...,p_m$ satisfy \eqref{Holder}.

There is a rich history of the study of square functions in harmonic analysis.  In \cite{Se}, Semmes studied the linear version ($m=1$) of the operators \eqref{theta}.  He proved that if $\theta_t$ satisfies \eqref{size}, \eqref{regy}, and there exists a para-accretive function $b$ such that $\Theta_t(b)=0$ for all $t>0$, then the bound \eqref{Lpbound} it's satisfied with $p=p_1=2$.  (For the definition of para-accretive see e.g. \cite{DJS}, \cite{Chr}, \cite{Se} or \cite{Hart1}.)   In fact the perspective of Semmes was a Besov type square function given in the multilinear setting by
\begin{align}
(f_1,...,f_m)\mapsto\(\int_0^\infty||\Theta_t(f_1,...,f_m)||_{L^p}^2\frac{dt}{t}\)^\frac{1}{2}.\label{Besov}
\end{align}
When $m=1$ and $p=p_1=2$ as in \eqref{Se}, the study of this Besov type square function \eqref{Besov} coincides with the study of \eqref{sqfunction}.  The Besov type square function point of view was carried to the multilinear setting by Maldonado in \cite{M} and Maldonado-Naibo in \cite{MN}, where the authors prove bounds of \eqref{Besov} on products of Besov and Lebesgue spaces under kernel conditions equivalent to \eqref{size} and \eqref{regy}, and $\Theta_t(1,f_2,...,f_m)=0$ for $t>0$.

In \cite{GO}, Grafakos-Oliveira proved the bound \eqref{Lpbound} for $p=2$ and $1\leq p_i\leq\infty$ for $i=1,...,m$ assuming \eqref{size}, \eqref{regy} and that there exist para-accretive functions $b_i$ for $i=1,...,m$ on $\R^n$ such that the cancelation condition
\begin{equation}
\Theta_t(b_1,...,b_m)=0\label{Theta(b)=0}
\end{equation}
holds.  In \cite{Hart1}, under similar size, regularity and cancellation conditions, Hart showed (in the discrete bilinear setting, but is easily extended to the $m$-linear setting) that \eqref{Lpbound} holds for $1<p,p_i<\infty$ for $i=1,...,m$, and under stronger size and regularity conditions for $1<p_i<\infty$ and $\frac{1}{2}<p<\infty$.  In \cite{Hart1} and \cite{GLMY}, Hart and Grafakos-Liu-Maldonado-Yang prove bounds of the square functions \eqref{sqfunction} and \eqref{Besov} on products of various spaces of smooth functions assuming \eqref{size}, \eqref{regy} and a variety of cancellation conditions.

In \cite{Chr}, Christ introduced the notion of a \emph{local Tb theorem} in the context of singular integrals, and applied this to estimates for the Cauchy integral on Lipschitz curves. He changed the existence of a (globally defined) para-accretive function where the operator vanishes, for the existence of a family of (locally defined) functions where you have some additional information about behavior of the operator. More recently, in \cite{Ho} Hofmann gave an analogous result for square functions based on some previous work by Auscher-McIntosh-Hofmann-Lacey-Tchamitchian on the Kato square root problem in \cite{AHLMcIT} (see also related work \cite{HMc} by Hofmann-McIntosh and \cite{HLMc} By Hofmann-Lacey-McIntosh) .

The principal result in this article is a extension of Hofmann's result to multilinear square functions, which we state now.
\begin{theorem}\label{t:main}
Let $\Theta_t$ and $S$ be defined as in \eqref{theta} and \eqref{sqfunction} where $\theta_t$ satisfies \eqref{size}-\eqref{regx}.  Suppose there exist $q_i,q>1$ for $i=1,...,m$ with $\frac{1}{q}=\sum_{i=1}^m\frac{1}{q_i}$ and functions $b_Q^i$ indexed by dyadic cubes $Q\subset\R^n$ for $i=1,...,m$ such that for every dyadic cube $Q$
\begin{align}
&\int|b_Q^i|^{q_i}\leq B_1|Q|\label{bsize}\\
&\frac{1}{B_2}\leq\left|\frac{1}{|Q|}\int_Q\prod_{i=1}^mb_Q^i(x)dx\right|\label{baccretive}\\
&\left|\frac{1}{|R|}\int_R\prod_{i=1}^mb_Q^i(x)dx\right|\leq B_3\prod_{i=1}^m\left|\frac{1}{|R|}\int_Rb_Q^i(x)dx\right|\label{bcompatible}\\
&\hspace{4cm}\text{ for all dyadic subcubes }R\subset Q\notag\\
&\int_Q\(\int_0^{\ell(Q)}|\Theta_t(b_Q^1,...,b_Q^i)(x)|^2\frac{dt}{t}\)^\frac{q}{2}dx\leq B_3|Q|.\label{thetacancel}
\end{align}
Then for all $1<p_i<\infty$ satisfying \eqref{Holder}
\begin{align}
||S(f_1,...,f_m)||_{L^2}\less\prod_{i=1}^m||f_i||_{L^{p_i}}\label{sqbound}
\end{align}
\end{theorem}

If $\{b_Q\}$ satisfies \eqref{bsize} and \eqref{baccretive}, we say that $\{b_Q\}$ is a pseudo-accretive system.   This definition of pseudo-accretive system is analogous to the one defined by Christ in \cite{Chr} in the linear case when restricted to the Euclidean setting.  More precisely, Christ defined a pseudo-accretive system to be a collection of functions $\{b_B\}$ indexed by all balls $B=B(x,r)\subset \R^n$ satisfying \eqref{bsize} and \eqref{baccretive} with $m=1$, $q=q_1=\infty$ and dyadic cubes $Q$ replaced with balls $B$.  We say that $\{b_Q^i\}$ for $i=1,...,m$ is an $m$-compatible, or just compatible, collection of pseudo-accretive systems if they satisfy \eqref{bsize}-\eqref{bcompatible}.  The proof of Theorem \ref{t:main} follows along the lines of the linear version in \cite{Ho}, with modifications to address difficulties that arise in the setting of multilinear operators.

We also prove that if the square function $S$ defined in \eqref{sqfunction} where the kernels of $\Theta_t$ satisfy \eqref{size},\eqref{regy} and \eqref{sqbound} for some indices $p,p_1,...,p_m$, then $S$ is also bounded $$L^{\infty}_c(\R^{n}) \times\cdots\times L^{\infty}_c(\R^{n})\to BMO$$ where $L^\infty_c$ is the set of $L^\infty$ functions with compact support. Note that $L^\infty_c$ is not a Banach space and $S$ is not a linear operator, so this bound does not mean that $S$ is continuous from $L^\infty\times\cdots\times L^\infty$ into $BMO$. This is simply an estimate for $f_1,...,f_m\in L^\infty_c$
$$||S(f_1,...,f_m)||_{BMO}\less\prod_{i=1}^m||f_i||_{L^\infty}$$
where the constant is independent of $f$ (and in particular the support $f_i$ for $i=1,...,m$). This means that we cannot use this bound to approximate $S(f_1,...,f_m)$ for $f_1,...,f_m\in L^\infty$, but the estimate is still useful for interpolation.  This will be discussed more in depth in section 4.

This permits us to prove the following generalization of the multilinear $T(1)$ theorem of Grafakos-Torres \cite{GT2} as a sort of multilinear version of the local Tb theorem of Christ in \cite{Chr}.

\begin{theorem}\label{t(1,1)}
Let $T$ be a continuous bilinear operator from $\S\times\cdots\times\S$ into $\S'$ with standard Calder\'on-Zygmund kernel $K$.  Suppose that $T\in WBP$ and there exist $2\leq q<\infty$ and $1<q_i<\infty$ with $\frac{1}{q}=\sum_{i=1}^m\frac{1}{q_i}$ and functions $b_Q^i$ indexed by dyadic cubes $Q\subset\R^n$ for $i=1,...,m$ that satisfy \eqref{bsize}-\eqref{bcompatible} and for all dyadic cubes $Q\subset\R^n$
\begin{align}
&\int_Q\(\int_0^{\ell(Q)}|Q_tT(P_tb_Q^1,...,P_tb_Q^m)(x)|^2\frac{dt}{t}\)^\frac{q}{2}dx\less|Q|\label{Tcancel}\\
&T^{*1}(1,...,1),...,T^{*m}(1,...,1)\in BMO.\label{Tstarcancel}
\end{align}
Then $T$ is bounded from $L^{p_1}\times\cdots\times L^{p_m}$ into $L^p$ for all $1<p_i<\infty$ such that \eqref{Holder} holds.  Here $P_t$ is an approximation to the identity and $Q_t$ a Littlewood-Paley-Stein projection operator both with $C_0^\infty$ convolution kernels.
\end{theorem}

To state \eqref{Tcancel} more precisely, we mean the following:  For any $\varphi,\psi\in C_0^\infty$ such that $\widehat\varphi(0)=1$ and $\widehat\psi(0)=0$, \eqref{Tcancel} holds for $P_tf=\varphi_t*f$ and $Q_tf=\psi_t*f$ where the constant is independent of the dyadic cube $Q$, but may depend on $\varphi$ and $\psi$.

The article is organized in the following way: In the next section we collect some results that will be useful in the proofs of the results stated above. In section 3, we prove the Theorem \ref{t:main} for $p=2$. In section 4, we precisely state and prove the $BMO$ endpoint estimate claimed above and complete the proof of Theorem \ref{t:main} for all $2\leq p<\infty$. In section 5, we prove the Theorem \ref{t(1,1)}.

The first author would like to thank... The second author would like to thank... The third author would like to thank...

\section{Preliminary Results}

In what follows $A\lesssim B$ means $A\le CB$ for some positive constant $C$. From this point on we will always work with smooth and compact supported functions, since the general result follows from density unless otherwise stated.

Define for $t>0$ the linear and multilinear dyadic average operators
\begin{align*}
&A_tf(x)=\frac{1}{|Q(x,t)|}\int_{Q(x,t)}f(x)\,dx,\\
&\mathbb A_t(f_1,...,f_m)(x)=\prod_{i=1}^mA_tf_i(x)
\end{align*}
where $Q(x,t)$ is the smallest dyadic cube containing $x$ with side length $\ell(Q)>t$. Define the linear and multilinear smooth approximation to the identity operators
\begin{align*}
&P_tf(x)=\int \varphi_{t}(x-y)f(y)\,dy,\\
&\mathbb P_t(f_1,...,f_m)(x)=\prod_{i=1}^mP_tf_i(x)
\end{align*}
where $\varphi\in C_0^\infty(\R^n)$ has integral $1$.

\begin{definition}
A positive measure $d\mu(x,t)$ on $\R^{n+1}_{+}=\{(x,t):x\in\R^n,\;t>0\}$  is called a   Carleson measure if
\begin{align}\label{10}
    \|d\mu\|_{\mathcal{C}}=\sup_Q\frac{1}{|Q|}d\mu(T(Q))<\infty\, ,
\end{align}
where the supremum is taken over all cubes $Q\subset\R^n$, $|Q|$ denotes the Lebesgue measure of the cube $Q$,   $T(Q)=Q\times (0,\ell(Q)]$ denotes the \emph{tent} over $Q$, and  $\ell(Q)$ is the side length of $Q$.
\end{definition}

We now state a result that is a multilinear version of the $T(1)$ theorem for square functions due to \cite{Hart1}, \cite{GLMY} and \cite{GO}.

\begin{proposition}[\cite{Hart1},\cite{GLMY},\cite{GO}]\label{SF}
Suppose that the kernel $\theta_{t}(x,y_{1},\dots,y_{m})$ satisfies \eqref{size}-\eqref{regy}.  If $\Theta_t(1,...,1)=0$ for $t>0$, then the square function defined in \eqref{sqfunction} satisfies the bound \eqref{Lpbound} for all $1<p,p_i< \infty$, $i=1,...,m$.
\end{proposition}

\remark Under extra size conditions on the kernel $\theta_{t}(x,y_{1},...,y_{m})$, i.e. if we require $N>2n$ in \eqref{size} and \eqref{regx}, we can apply the vector-valued Calder\'on-Zygmund theory developed in \cite{Hart1} to extend the theorem above to the complete quasi-Banach case, that is, with $1/2<p\leq1$.

The following result relates Carleson measures and a special kind of multilinear operator that will be useful for us.  An important tool in the proof of the above theorems is the following multilinear version of a theorem of Christ and Journ\'e \cite{CJ}.
\begin{proposition}\label{p:Carlesonbound}
Assume $\Theta_t$ and $S$ are defined as in \eqref{theta} and \eqref{sqfunction} where $\theta_t$ satisfies \eqref{size}-\eqref{regx}.  If $\Theta_t$ satisfies the Carleson measure estimate
\begin{align}
\int_Q\int_0^{\ell(Q)}|\Theta_t(1,...,1)(x)|^2\frac{dt\,dx}{t}\less|Q|\label{Carlesonbound}
\end{align}
for all cubes $Q\subset\R^n$, then \eqref{Lpbound} holds when $p=2$ and $1<p_i<\infty$ for $i=1,...,m$.
\end{proposition}
\begin{proof}
We decompose $\Theta_t=\Theta_t-M_{\Theta_t(1,...,1)}\mathbb P_t+M_{\Theta_t(1,...,1)}\mathbb P_t$, where $M_b$ is the operator defined as pointwise multiplication by $b$.  It is clear that $\Theta_t-M_{\Theta_t(1,...,1)}\mathbb P_t$ satisfies \eqref{size}, \eqref{regy} and $\Theta_t(1,...,1)-M_{\Theta_t(1,...,1)}\mathbb P_t(1,...,1)=0$.  Then by proposition \ref{SF}, it follows that the square function associated to $\Theta_t-M_{\Theta_t(1,...,1)}\mathbb P_t$ is bounded for all $1<p,p_1,...,p_m<\infty$.  Using this bound and that $|\Theta_t(1,...,1)(x)|^2\frac{dt\,dx}{t}$ is a Carleson measure (by assumption)
\begin{align*}
||S(f_1,...,f_m)||_{L^2}&\leq\left|\left|\(\int_0^\infty|\Theta_t(f_1,...,f_m)-M_{\Theta_t(1,...,1)}\mathbb P_t(f_1,...,f_m)|^2\frac{dt}{t}\)^\frac{1}{2}\right|\right|_{L^2}\\
&\hspace{2.5cm}+\prod_{i=1}^m\(\int_{\R^{n+1}_+}|P_tf_i(x)|^{p_i}|\Theta_t(1,...,1)(x)|^2\frac{dx\,dt}{t}\)^\frac{1}{p_i}\\
&\less\prod_{i=1}^m||f_i||_{L^{p_i}}.
\end{align*}
The final inequality uses the well-known Carleson measure estimate:  If $d\mu(x,t)$ is a Carleson measure, then $P:f\mapsto P_tf(x)$ is bounded from $L^q(\R^n)$ into $L^q(\R^{n+1}_+,d\mu)$ for $1<q<\infty$.
\end{proof}

The next result allows us to compare the multilinear dyadic averaging operators $\mathbb A_t$ and the multilinear smooth approximation to the identity operators $\mathbb P_t$. This comparison principle will be important in the proof of Theorem \ref{t:main}.  This is a particular case of a multilinear version of a result of Duoandikoetxea-Rubio de Francia in \cite{DRdF}.

\begin{proposition}\label{p:error}
Let $A_t$, $P_t$, $\mathbb A_t$ and $\mathbb P_t$ be as above.  Then for all $1<p_i<\infty$, $i=1,...,m$, we have the bound
\begin{align*}
\left|\left|\(\int_0^\infty|\mathbb A_t(f_1,...,f_m)-\mathbb P_t(f_1,...,f_m)|^2\frac{dt}{t}\)^\frac{1}{2}\right|\right|_{L^p}\less\prod_{i=1}^m||f_i||_{L^{p_i}}.
\end{align*}
Note that this even holds for $\frac{1}{m}<p<\infty$ as long as $1<p_i<\infty$.
\end{proposition}
\begin{proof}
Define for $j=1,...,m$
\begin{align*}
\mathbb E_t^j(f_1,...,f_m)=\(\prod_{i=1}^{j-1}A_tf_i\)(A_tf_j-P_tf_j)\(\prod_{i=j+1}^mP_tf_i\).
\end{align*}
Here we use the convection that $\prod_{i=1}^0A_i=\prod_{i=m+1}^mP_t=1$.  Then we have the following decomposition by successively adding and subtracting the term $A_tf_1\cdots A_tf_{j}P_tf_{j+1}\cdots P_tf_m$
\begin{align*}
\mathbb A_t(f_1,...,f_m)-\mathbb P_t(f_1,...,f_m)&=\mathbb E_t^1(f_1,...,f_m)+A_tf_1\(\prod_{i=2}^mA_tf_i-\prod_{i=2}^mP_tf_i\)\\
&=\sum_{j=1}^2\mathbb E_t^j(f_1,...,f_m)+A_tf_1A_tf_2\(\prod_{i=3}^mA_tf_i-\prod_{i=3}^mP_tf_i\)\\
&=\sum_{j=1}^m\mathbb E_t^j(f_1,...,f_m).
\end{align*}
It is a standard argument to show that  $\sup_{t>0}|P_tf(x)|\less Mf(x)$ where $M$ is the Hardy-Littlewood maximal function, and the same inequality holds replacing $P_t$ with $A_t$.  Then we use the linear bound of $A_t-P_t$ which was proved by Duoandikoetxea-Rubio de Francia \cite{DRdF}
\begin{align*}
\left|\left|\(\int_0^\infty|\mathbb E_t^j(f_1,...,f_m)|^2\frac{dt}{t}\)^\frac{1}{2}\right|\right|_{L^p}&\less\left|\left|\(\int_0^\infty|(A_t-P_t)f_j|^2\frac{dt}{t}\)^\frac{1}{2}\prod_{i\neq j}Mf_i\right|\right|_{L^p}\\
&\leq\left|\left|\(\int_0^\infty|(A_t-P_t)f_j|^2\frac{dt}{t}\)^\frac{1}{2}\right|\right|_{L^{p_j}}\prod_{i\neq j}||Mf_i||_{L^{p_i}}\\
&\less\prod_{i=1}^m||f||_{L^{p_i}}.
\end{align*}
The square function bound for $\mathbb A_t-\mathbb P_t$ easily follows.
\end{proof}

\section{Proof of Theorem \ref{t:main} with p=2}

We proceed by reducing our arguments to the dyadic case.  Using dyadic covering properties it is easy to see that if \eqref{Carlesonbound} holds for all dyadic cubes, then \eqref{Carlesonbound} holds for all cubes $Q$ with at slightly larger constant.  In the following, we prove \eqref{Carlesonbound} for dyadic cubes to conclude \eqref{Lpbound} for $p=2$, and then proceed with other techniques in the next section.

\subsection{Decomposition of Dyadic Cubes}

We start with a proposition similar to one used in \cite{Ho}, applied to $m$ collections of pseudo-accretive systems $\{b_Q^i\}$ for $i=1,...,m$.

\begin{proposition}\label{p:decomp}
Given an $m$ linear compatible systems of functions $\{b_Q^i\}$ indexed by dyadic cubes for $i=1,...,m$ satisfying \eqref{bsize}-\eqref{thetacancel}, there exists a collection of non-overlapping dyadic subcubes of $Q$, $\{Q_k\}$, and $\eta\in(0,1)$
\begin{align}
&\sum_k|Q_k|<(1-\eta)|Q|,\label{eta}
\end{align}
where $\eta$ does not depend on $Q$, and for $t>\tau_Q(x)$ and $x\in Q$
\begin{align}
\frac{1}{2B_2B_3}<\prod_{i=1}^m|A_tb_Q^i(x)|\;\;\text{(here WLOG we assume that $B_2,B_3\geq1$)}\label{average}
\end{align}
where
\begin{align}
\tau_Q(x)&=\left\{\begin{array}{ll}\ell(Q_k)&x\in Q_k\\0&x\in E\end{array}\right.\label{tau}\\
E&=Q\backslash\bigcup_kQ_k.\notag
\end{align}
\end{proposition}

\begin{proof}
Fix a dyadic cube $Q\subset\R^n$ and define
$$a=\frac{1}{|Q|}\int_Q\prod_{i=1}^mb_Q^i(x)dx$$
which satisfies $|a|\geq\frac{1}{B_2}$, where $B_2$ is from \eqref{baccretive}.  Now choose from the dyadic children of $Q$ the cubes that are maximal with respect to the property
\begin{align*}
Re\[\frac{1}{a|Q_j|}\int_{Q_j}\prod_{i=1}^mb_Q^i(x)dx\]\leq\frac{1}{2},
\end{align*}
i.e. $Q_j\subset Q$ is the largest dyadic cube such that the above inequality holds.  By the properties of dyadic cubes, these maximal cubes are non-overlapping.  This stopping time criterion well defines a collection of cubes since
\begin{align*}
Re\[\frac{1}{a|Q|}\int_{Q}\prod_{i=1}^mb_Q^i(x)dx\]=1
\end{align*}
If $x\in Q_k$ for some $k$ and $t>\tau_Q(x)$, then using \eqref{bcompatible}
\begin{align*}
|\mathbb A_t(b_Q^1,...,b_Q^m)(x)|&=\prod_{i=1}^m\left|\frac{1}{|Q(x,t)|}\int_{Q(x,t)}b_Q^i(y)dy\right|\\
&\geq \frac{1}{B_3}\left|\frac{1}{|Q(x,t)|}\int_{Q(x,t)}\prod_{i=1}^mb_Q^i(y)dy\right|\\
&\geq \frac{|a|}{B_3}Re\(\frac{1}{a|Q(x,t)|}\int_{Q(x,t)}\prod_{i=1}^mb_Q^i(y)dy\)\\
&\geq\frac{1}{2B_2B_3}
\end{align*}
Also if $x\in E$, then again using \eqref{bcompatible} and by the stopping time criterion it follows that
\begin{align*}
|\mathbb A_t(b_Q^1,...,b_Q^m)(x)|&=\prod_{i=1}^m\left|\frac{1}{|Q(x,t)|}\int_{Q(x,t)}b_Q^i(y)dy\right|\\
&\geq\frac{1}{B_3}\left|\frac{1}{|Q(x,t)|}\int_{Q(x,t)}\prod_{i=1}^mb_Q^i(y)dy\right|\\
&\geq\frac{|a|}{B_3}Re\(\frac{1}{a|Q(x,t)|}\int_{Q(x,t)}\prod_{i=1}^mb_Q^i(y)dy\)\\
&\geq\frac{1}{2B_2B_3}.
\end{align*}
Now we also have for $i=1,...,m$ that
\begin{align*}
|Q|&=Re\[\frac{1}{a}\int_Q\prod_{i=1}^mb_Q^i(x)dx\]\\
&\leq\sum_kRe\[\frac{1}{a}\int_{Q_k}\prod_{i=1}^mb_Q^i(x)dx\]+\int_E\left|\prod_{i=1}^mb_Q^i(x)\right|dx\\
&\leq\frac{1}{2}\sum_k|Q_k|+|E|^\frac{1}{q'}\(\int_E\left|\prod_{i=1}^mb_Q^i(x)\right|^qdx\)^\frac{1}{q}\\
&\leq\frac{1}{2}|Q|+|E|^\frac{1}{q'}\prod_{i=1}^m\(\int_Q|b_Q^i(x)|^{q_i}dx\)^\frac{1}{q_i}\\
&\leq\frac{1}{2}|Q|+B_1^m|E|^\frac{1}{q'}|Q|^\frac{1}{q}.
\end{align*}
It follows that $\eta|Q|<|E|$ where we may take $\eta=\frac{1}{(2B_1^m)^{q'}}\in(0,1)$.
\end{proof}

\subsection{Reduction to Carleson Estimates}

We pause for a moment to discuss the strategy of the remainder of the proof of Theorem \ref{t:main} for $p=2$.  By Proposition \ref{p:Carlesonbound} and the discussion at the beginning of this section, it is sufficient to show that the estimate \eqref{Carlesonbound} holds for dyadic cubes.  In order to show this, we prove an intermediate estimate:  For all dyadic cubes $Q\subset\R^n$
\begin{align}
\int_Q\(\int_{\tau_Q(x)}^{\ell(Q)}|\Theta_t(1,...,1)(x)|^2\frac{dt}{t}\)^\frac{q}{2}dx\leq C|Q|.\label{truncate}
\end{align}
The remainder of this section is dedicated to proving \eqref{truncate}, and the next section completes the proof of Theorem \ref{t:main} for $p=2$ by proving \eqref{Carlesonbound} from the reduction in this section.

\begin{proposition}\label{p:truncate}
For all dyadic cubes $Q\subset\R^n$, \eqref{truncate} holds with $\tau_Q$ defined in \eqref{tau}
\end{proposition}
\begin{proof}
We have from Proposition \ref{p:decomp} that $|\mathbb A_t(b_Q^1,...,b_Q^m)(x)|\geq\frac{1}{2B_2^mB_3}$, so it follows that
\begin{align*}
&\int_Q\(\int_{\tau_Q(x)}^{\ell(Q)}|\Theta_t(1,...,1)(x)|^2\frac{dt}{t}\)^\frac{q}{2}dx\\
&\hspace{2cm}\leq2B_2^mB_3\int_Q\(\int_{\tau_Q(x)}^{\ell(Q)}|\Theta_t(1,...,1)(x)\mathbb A_t(b_Q^1,...,b_Q^m)(x)|^2\frac{dt}{t}\)^\frac{q}{2}dx.
\end{align*}
Now we consider the operator $M_{\Theta_t(1,...,1)}\mathbb A_t(f_1,...,f_m)$, which we decompose in the following way
\begin{align*}
M_{\Theta_t(1,...,1)}\mathbb A_t&=M_{\Theta_t(1,...,1)}(\mathbb A_t-\mathbb P_t)+(M_{\Theta_t(1,...,1)}\mathbb P_t-\Theta_t)+\Theta_t\\
&=R_t^{(1)}+R_t^{(2)}+\Theta_t
\end{align*}
By Proposition \ref{p:error}, it follows that
\begin{align*}
\int\(\int_0^\infty|R_t^{(1)}(b_Q^1,...,b_Q^m)|^2\frac{dt}{t}\)^\frac{q}{2}dx&\less\prod_{i=1}^m||b_Q^i||_{L^{q_1}}^q\less|Q|.
\end{align*}
Using Proposition \ref{SF}, it follows that the $R_t^{(2)}$ term is controlled as desired
\begin{align*}
\int\(\int_0^\infty|R_t^{(2)}(b_Q^1,...,b_Q^m)|^2\frac{dt}{t}\)^\frac{q}{2}dx&\less|Q|,
\end{align*}
and by hypothesis \eqref{thetacancel},
\begin{align*}
\int_Q\(\int_0^{\ell(Q)}|\Theta_t(b_Q^1,...,b_Q^m)|^2\frac{dt}{t}\)^\frac{q}{2}dx&\leq B_3|Q|.
\end{align*}
Then we may choose $C$ independent of $Q$ such that \eqref{truncate} holds.
\end{proof}

\subsection{End of the Proof}

Finally we use the reduction from the previous section to complete the proof of Theorem \ref{t:main}.

\begin{lemma}\label{l:weakestimate}
There exist $N>0$ and $\beta\in(0,1)$ such that for every dyadic cube $Q$
\begin{align}
|\{x\in Q:g_Q(x)>N\}|\leq(1-\beta)|Q|\label{badset}
\end{align}
where
\begin{align}
g_Q(x)=\(\int_0^{\ell(Q)}|\Theta_t(1,...,1)(x)|^2\frac{dt}{t}\)^\frac{1}{2}\label{gQ}
\end{align}
where $\tau_Q(x)$ is defined as in \eqref{tau}.
\end{lemma}
\begin{proof}
Fix a dyadic cube $Q\subset\R^n$, and define for $N>0$
$$\Omega_N=\{x\in Q:g_Q(x)>N\}$$
Let $Q_k$ and $E$ be as in Proposition \ref{p:decomp}, without loss of generality take $N,C>1$, and using Chebychev's inequality it follows that
\begin{align*}
|\Omega_N|&\leq\sum_k|Q_k|+\left|\left\{x\in E:g_Q(x)>N\right\}\right|\\
&\leq(1-\eta)|Q|+\frac{C}{N^q}|Q|
\end{align*}
where $C$ is chosen from \eqref{truncate} in Proposition \ref{p:truncate} as discussed above.  Now fix $N$ large enough so that $\frac{C}{N^q}<\eta/2$.  Then \eqref{badset} easily follows
\begin{align*}
|\Omega_N|&\leq(1-\eta)|Q|+\frac{C}{N^q}|Q|<(1-\beta)|Q|
\end{align*}
where $\beta=\frac{\eta}{2}>0$.
\end{proof}
We can finally prove the main theorem for $p=2$
\begin{proof}
Fix $\epsilon\in(0,1)$ and define for dyadic cube $Q\subset\R^n$ with $\ell(Q)>\epsilon$
$$g_{Q,\epsilon}(x)=\(\int_{\epsilon}^{\min(1/\epsilon,\ell(Q))}|\Theta_t(1,...,1)(x)|^2\frac{dt}{t}\)^\frac{1}{2}$$
and $g_{Q,\epsilon}=0$ if $\ell(Q)\leq\epsilon$.  Also define
$$K(\epsilon)=\sup_Q\frac{1}{|Q|}\int_Qg_{Q,\epsilon}(x)dx$$
where the supremum is over all dyadic cubes.  Fix a dyadic cube $Q$ and define
$$\Omega_{N,\epsilon}=\{x\in Q:g_{Q,\epsilon}(x)>N\}.$$
Note that $g_{Q,\epsilon}$ is defined depending only on the cube $Q$ and $\epsilon$, completely independent of $Q_k$, $\tau_Q(x)$ and $\eta$.  It follows from \eqref{regx} that $\Theta_t(1,...,1)(x)$ is $\gamma$-H\"older continuous and hence so is $g_{Q,\epsilon}$ (with constant depending on $\epsilon$).
\begin{align*}
|g_{Q,\epsilon}(x)-g_{Q,\epsilon}(x')|^2&\less\int_{\epsilon}^{\min(1/\epsilon,\ell(Q))}\(\int_{\R^{mn}}\frac{t^{-mn}(t^{-1}|x-x'|)^\gamma}{\prod_{i=1}^m(1+t^{-1}|x-y_i|)^{N}}\prod_{i=1}^mdy_i\)^2\frac{dt}{t}\\
&\less\epsilon^{-2-\gamma}|x-x'|^\gamma.
\end{align*}
Then $g_{Q,\epsilon}$ is continuous, $\Omega_{N,\epsilon}$ is open, and so we may make the Whitney decomposition $Q_j$ of $\Omega_{N,\epsilon}$.  That is there exists a collection of cubes $\{Q_j\}$ such that
\begin{align}
&\bigcup_jQ_j=\Omega_{N,\epsilon}\\
&\sqrt n\ell(Q_j)\leq dist(Q_j,\Omega^c)\leq4\sqrt n\ell(Q_j)\\
&\partial Q_j\cap Q_k\neq\emptyset\Longrightarrow\frac{1}{4}\leq\frac{\ell(Q_j)}{\ell(Q_k)}\leq4\\
&\text{ Given a cube, there are at most }12^n\text{ that touch it}.
\end{align}
Then if $F_{N,\epsilon}=Q\backslash\Omega_{N,\epsilon}$
\begin{align*}
\int_Qg_{Q,\epsilon}^2(x)dx&=\int_{F_{N,\epsilon}}g_{Q,\epsilon}^2(x)dx+\sum_j\int_{Q_j}g_{Q,\epsilon}^2(x)dx\\
&\hspace{-1cm}\leq N^2|Q|+\sum_j\int_{Q_j}\int_{\epsilon}^{\min(1/\epsilon,\ell(Q_j))}|\Theta_t(1,...,1)(x)|^2\frac{dt\,dx}{t}\\
&\hspace{2cm}+\sum_j\int_{Q_j}\int_{\max(\epsilon,\ell(Q_j)}^{\min(1/\epsilon,\ell(Q))}|\Theta_t(1,...,1)(x)|^2\frac{dt\,dx}{t}\\
&\hspace{-1cm}\leq N^2|Q|+K(\epsilon)\sum_j|Q_j|+\sum_j\int_{Q_j}\int_{\max(\epsilon,\ell(Q_j))}^{\min(1/\epsilon,\ell(Q))}|\Theta_t(1,...,1)(x)|^2\frac{dt\,dx}{t}\\
&\hspace{-1cm}\leq N^2|Q|+K(\epsilon)\eta |Q|+\sum_j\int_{Q_j}\int_{\max(\epsilon,\ell(Q_j))}^{\min(1/\epsilon,\ell(Q))}|\Theta_t(1,...,1)(x)|^2\frac{dt\,dx}{t}.
\end{align*}
To control the last term, since $Q_j$ is a Whitney decomposition, there exists $x_j\in F_{N,\epsilon}$ such that
$$dist(x_j,Q_j)\leq(4\sqrt n+1)\ell(Q_j).$$
We have for $x\in Q_j$
\begin{align*}
|\Theta_t(1,...,1)(x)-\Theta_t(1,...,1)(x_j)|&\leq\int_{\R^{mn}}|\theta_t(x,y_1,...,y_m)-\theta_t(x_j,y_1,...,y_m)|\prod_{i=1}^mdy_i\\
&\less\int_{\R^{mn}}\frac{t^{-2n}(t^{-1}|x-x_j|)^\gamma}{\prod_{i=1}^m(1+t^{-1}|x-y_i|)^N}\prod_{i=1}^mdy_i\\
&\less(t^{-1}\ell(Q_j))^\gamma.
\end{align*}
So choose $c_1$ which depends only on the dimension such that the inequality
$$|\Theta_t(1,...,1)(x)-\Theta_t(1,...,1)(x_j)|\leq c_1(t^{-1}\ell(Q_j))^\gamma$$
holds for all $x\in Q_j$.  Then
\begin{align*}
&\sum_j\int_{Q_j}\int_{\max(\epsilon,\ell(Q_j))}^{\min(1/\epsilon,\ell(Q))}|\Theta_t(1,...,1)(x)|^2\frac{dt\,dx}{t}\\
&\hspace{.5cm}\leq\sum_j\int_{Q_j}\int_{\max(\epsilon,\ell(Q_j))}^{c_1\ell(Q_j)}|\Theta_t(1,...,1)(x)|^2\frac{dt\,dx}{t}\\
&\hspace{1.5cm}+\sum_j\int_{Q_j}\int_{c_1\ell(Q)}^{\min(\ell(Q),1/\epsilon)}|\Theta_t(1,...,1)(x_j)|^2\frac{dt\,dx}{t}\\
&\hspace{2cm}+\sum_j\int_{Q_j}\int_{c_1\ell(Q)}^{\min(\ell(Q),1/\epsilon)}|\Theta_t(1,...,1)(x)-\Theta_t(1,...,1)(x_j)|^2\frac{dt\,dx}{t}\\
&\hspace{.5cm}=I+II+III.
\end{align*}
We have that
\begin{align*}
I&\leq||\Theta_t(1,...,1)||_{L^\infty}^2\sum_j\int_{Q_j}\int_{\ell(Q_j)}^{c_1\ell(Q_j)}\frac{dt\,dx}{t}\less c_1\sum_j|Q_j|\less|Q|.
\end{align*}
Since $x_j\in F_{N,\epsilon}$ and $g_{Q,\epsilon}(x_j)\leq N$, it follows that
\begin{align*}
II&\leq\sum_j\int_{Q_j}\int_0^{\ell(Q)}|\Theta_t(1,...,1)(x_j)|^2\frac{dt\,dx}{t}=\sum_j|Q_j|g_{Q,\epsilon}(x_j)^2\less N^2|Q|.
\end{align*}
For all $x\in Q_j$, $|\Theta_t(1,...,1)(x)-\Theta_t(1,...,1)(x_j)|\leq c_1(t^{-1}\ell(Q_j))^\alpha$, so
\begin{align*}
III&\less\sum_j\int_{Q_j}\int_{c_1\ell(Q_j)}^\infty(t^{-1}\ell(Q_j))^\alpha \frac{dt\,dx}{t}\less\sum_{j}|Q_j|\leq|Q|.
\end{align*}
Therefore $K(\epsilon)\leq C(1+N^2)+(1-\beta)K(\epsilon)$ and hence
$$K(\epsilon)\leq\frac{C(1+N^2)}{\beta}.$$
Therefore
\begin{align*}
\frac{1}{|Q|}\int_Q\int_0^{\ell(Q)}|\Theta_t(1,...,1)(x)|^2\frac{dt\,dx}{t}&=\sup_{0<\epsilon<1}\sup_{\ell(Q)>\epsilon}\frac{1}{|Q|}\int_Q\int_\epsilon^{\ell(Q)}|\Theta_t(1,...,1)(x)|^2\frac{dt\,dx}{t}\\
&=\sup_{0<\epsilon<1}K(\epsilon)\\
&\leq\frac{C(1+N^2)}{\beta}.
\end{align*}
Hence $|\Theta_t(1,...,1)(x)|^2\frac{dt\,dx}{t}$ is a Carleson measure and by Proposition \ref{Carlesonbound} the square function bound \eqref{Lpbound} holds with constant $C(1+N^2)/\beta$ for $p=2$ and $1<p_1,...,p_m<\infty$.
\end{proof}

This proves theorem \ref{t:main} for $p=2$.  In the following section we prove that this we can strengthen the conclusion of theorem can be strengthened to conclude that \eqref{Lpbound} holds for all $2\leq p<\infty$ and $1<p_1,...,p_m<\infty$, but first we make some remarks on compatible pseudo-accretive systems.

\subsection{A Comment on Compatible Pseudo-Accretive Systems}

The purpose of this discussion is to better understand the conditions \eqref{baccretive} and \eqref{bcompatible} through various examples.  In the first example we construct a class of non-trivial classes of compatible pseudo-accretive systems.

\subsubsection{Example \ref{ex1}}\label{ex1}
Suppose there exists $\epsilon>0$ such that $\epsilon\leq b_Q^i(x)\leq\epsilon^{-1}$ for a.e. $x\in Q$, all dyadic cubes $Q\subset\R^n$ and each $i=1,...,m$, then \eqref{baccretive} and \eqref{bcompatible} hold as well,
\begin{align*}
\epsilon^m\leq\left|\frac{1}{|R|}\int_R\prod_{i=1}^mb_Q^i(x)dx\right|&\leq\epsilon^{-m}\leq\epsilon^{-2m}\prod_{i=1}\left|\frac{1}{|R|}\int_Rb_Q^i(x)dx\right|.
\end{align*}
Notice that this is a uniform condition for $b_Q^i$.  That is there is no dependence between the functions, as long as they are each in this class of functions.  This class of functions includes many commonly used functions.  For example, the following functions defined for each dyadic cube $Q\subset\R^n$ satisfy $\epsilon<b_Q<\epsilon^{-1}$ uniformly on the cube $Q$ for some $\epsilon$.
\begin{align*}
&\text{Characteristic functions:}& &b_Q(x)=\chi_Q(x)&\\
&\text{Gaussian functions:}& &b_Q(x)=e^{-\frac{|x-x_Q|^2}{\ell(Q)^2}}&\\
&\text{Poisson kernels:}& &b_Q(x)=\frac{\ell(Q)^{n+1}}{(\ell(Q)^2+|x-x_Q|^2)^\frac{n+1}{2}}.&
\end{align*}

\subsubsection{Example \ref{ex2}}\label{ex2}
Consider the pseudo-accretive systems on $\R$ for dyadic cubes $Q_{j,k}=[j2^{-k},(j+1)2^{-k})$ defined
\begin{align*}
&b_{Q_{j,k}}^1=b_{j,k}^1=\chi_{[j2^{-k},(j+3/4)2^{-k})}-\chi_{[(j+3/4)2^{-k},(j+1)2^{-k})}\\
&b_{Q_{j,k}}^2=b_{j,k}^2=\chi_{[(j+1/4)2^{-k},(j+1)2^{-k})}-\chi_{[j2^{-k},(j+1/4)2^{-k})}
\end{align*}
It follows that $b_{j,k}^i$ satisfies \eqref{baccretive} for $i=1,2$ by a quick computation
\begin{align*}
\frac{1}{|Q_{j,k}|}\int_{Q_{j,k}}b_{j,k}^1(x)dx=\frac{1}{|Q_{j,k}|}\int_{Q_{j,k}}b_{j,k}^2(x)dx=\frac{1}{2}.
\end{align*}
It is a bit more complicated to see that $b_{j,k}^1,b_{j,k}^2$ satisfy \eqref{bcompatible}, but it does hold:  For $R=Q_{j,k}$, it follows that the left hand side of \eqref{bcompatible} is zero so the inequality holds.  Now if $R\subset Q_{j,k}$ is any dyadic subcube contained in $[j2^{-k},(j+1/2)2^{-k})$, then $b_{j,k}^1=1$ on $R$ and
\begin{align*}
\frac{1}{|R|}\int_Rb_{j,k}^1(x)b_{j,k}^2(x)dx&=\frac{1}{|R|}\int_Rb_{j,k}^2(x)dx=\prod_{i=1}^2\frac{1}{|R|}\int_Rb_{j,k}^i(x)dx.
\end{align*}
A symmetric argument holds when $R\subset[(j+1/2)2^{-k},(j+1)2^{-k})$.  Therefore $b_{j,k}^1,b_{j,k}^2$ are compatible pseudo-accretive systems.  This example is especially interesting because there are subcubes where $b_{j,k}^i$ has mean zero, $b_{j,k}^1\cdot b_{j,k}^2$ has mean zero, but the particular structure of these functions allow for \eqref{baccretive} and \eqref{bcompatible} hold.

\subsubsection{Example \ref{ex3}}\label{ex3}
There exist pseudo-accretive systems that are not compatible.  To construct such a system, we consider the bilinear setting and $\R$.  Consider the cube $Q=[0,2]\subset\R$ and define
\begin{align*}
b_Q=b_Q^1(x)=b_Q^2(x)=(x-\frac{1}{2})\chi_{[0,2]}(x)
\end{align*}
We have that $b_Q^i$ satisfy \eqref{baccretive} for $i=1,2$
\begin{align*}
\left|\int_{[0,2]}b_Q(x)dx\right|&=1,
\end{align*}
but if we consider the dyadic subcube $[0,1]\subset[0,2]$, the functions violate \eqref{bcompatible}
\begin{align*}
&\left|\int_{[0,1]}b_Q^1(x)b_Q^2(x)dx\right|=\int_0^1(x^2-x+\frac{1}{2})dx=\frac{1}{3},\\
&\prod_{i=1}^2\left|\int_{[0,1]}b_Q^i(x)dx\right|=\(\int_0^1(x-\frac{1}{2})dx\)^2=0.
\end{align*}
Here it is apparent that the failure of condition \eqref{bcompatible} is caused by the cancellation of $b_Q^1$ and $b_Q^2$ in the same location.

From Examples \ref{ex1} and \ref{ex2}, we can see that there are non-trivial compatible pseudo-accretive system, even some with cancellation on dyadic subcubes.  Example \ref{ex3} demonstrates that there are pseudo-accretive systems that aren't compatible, and furthermore the functions in Example \ref{ex3} fail to satisfy the compatibility condition \eqref{bcompatible} because they have cancellation behavior in the same location.

\section{Extending Square Function Bounds}

In this section we prove a multilinear $BMO$ bound and use it as an endpoint for interpolation.  More precisely, we prove the following $L^\infty_c\times\cdots\times L^\infty_c\rightarrow BMO$ bound.

\begin{theorem}\label{t:BMO}
Suppose $\Theta_t$ satisfies \eqref{size}-\eqref{regx} and the square function $S$ associated to $\Theta_t$ is bounded from $L^{p_1}\times\cdots\times L^{p_m}$ into $L^p$ for some $1\leq p,p_i\leq\infty$ that satisfy \eqref{Holder}.  Then for all $f_1,...,f_m\in L_c^\infty$
\begin{align}
||S(f_1,...,f_m)||_{BMO}\less\prod_{i=1}^m||f_i||_{L^\infty}\label{BMO}
\end{align}
where the constant is independent of $f_i$ (and in particular the support of $f_i$) for $i=1,...,m$.
\end{theorem}

This is essentially a square function version of a corresponding result for multilinear Calder\'on-Zygmund operators from Grafakos-Torres \cite{GT1}:  If a multilinear Calder\'on-Zygmund operator $T$ is bounded from $L^{p_1}\times\cdots\times L^{p_m}$ into $L^{p}$ for some $1<p,p_1,...,p_m<\infty$, then $T$ is bounded from $L^\infty_c\times\cdots\times L^\infty_c$ into $BMO$.  In \cite{GT1}, the authors prove this using an inductive argument by reducing the $m$ linear case to the $m-1$ linear one.  Here we present a direct multilinear proof adapted from the classical linear version due to Spanne \cite{Sp}, Peetre \cite{P} and Stein \cite{S2}, but prior to this proof we briefly discuss why we don't conclude here that $S$ is bounded from $L^\infty\times\cdots\times L^\infty$ into $BMO$.

In \cite{GT1}, the authors also conclude that if an $m$-linear Calder\'on-Zygmund operator $T$ is bounded, then $T$ is bounded from $L^\infty\times\cdots\times L^\infty$ into $BMO$ estimate.  One difficultly in this problem is that $T$ is not necessarily even defined for $f_1,...,f_m\in L^\infty$.  So one must define $T$ for $f_1,...,f_m\in L^\infty$, and the definition for such functions must be consistent with the given definition of $T$ in the case that $f_i\in L^{p_i}\cap L^\infty$.  As it turns out (see \cite{GT1}), it is reasonable to define for $f_1,...,f_m\in L^\infty$
\begin{align*}
T(f_1,...,f_m)&=\lim_{R\rightarrow\infty}T(f_1\chi_{B(0,R)},...,f_m\chi_{B(0,R)})\\
&\hspace{3cm}-\int_{|y_i|>1}K(0,y_1,...,y_m)\prod_{i=1}^mf_i(y_i)\chi_{B(0,R)}(y_i)dy_i
\end{align*}
where the limit is taken in the dual of $C_{c,0}^\infty(\R^n)$.  Here $C_{c,0}^\infty(\R^n)$ is the collection of all smooth compactly supported functions with mean zero.  As expected, this well defines $T$ on $L^\infty\times\cdots\times L^\infty$ modulo a constant, which is permissible as an element of $BMO$.  The existence of this limit follows from the linearity and kernel estimates of $T$.  Along with the $L^\infty_c\times\cdots\times L^\infty_c\rightarrow BMO$ estimate for $T$, the existence of this limit implies that $T$ is bounded from $L^\infty\times\cdots\times L^\infty$ into $BMO$.

Morally we expect the same estimates for the square function $S$ defined in \eqref{sqfunction} as we have been proved for a multilinear Calder\'on-Zygmund operator $T$.  Despite the estimate for $S$ on $L^\infty_c\times\cdots\times L^\infty_c$, we are unable to make the same boundedness conclusion on $L^\infty\times\cdots\times L^\infty$ for $S$ as can be made for $T$.  The reason for this essentially comes down to the fact that $S$ is not a linear operator.  If one tries to mimic the proof from \cite{GT1} replacing $T$ with $S$, the above limit does not necessarily exist.  So the problem becomes finding a suitable definition for $S$ on $L^\infty\times\cdots\times L^\infty$, as the classical definition does not necessarily exist (at least using the same proof techniques).  Another approach to define $S$ on $L^\infty\times\cdots\times L^\infty$ is to view $\Theta_t$ as an $m$-linear operator taking values in $L^2(\R_+,\frac{dt}{t})$.  In this case one may be able
 to define $S$ as a weak limit of an appropriate space of smooth functions taking values in $L^2(\R_+,\frac{dt}{t})$.  Since we only need the previous estimate for compactly supported functions to prove our interpolation theorem, we will not pursue this approach here.

\begin{proof}
Assume that $f_i\in L_c^\infty$ for $i=1,...,m$ and $B=B(x_B,R)\subset\R^n$ is a ball for some $R>0$ and $x_B\in\R^n$.  Define
\begin{align*}
c_B=\(\int_0^\infty|\Theta_t(f_1,...,f_m)(x_B)-\Theta_t(f_1\chi_{2B},...,f_m\chi_{2B})(x_B)|^2\frac{dt}{t}\)^\frac{1}{2},
\end{align*}
which exists since $f_1,...,f_m\in L^p$ for all $1\leq p\leq\infty$ since we have assumed that $f_1,...,f_m$ are compactly supported.  Then it follows that
\begin{align*}
\int_B\left|S(f_1,...,f_m)(x)-c_B\right|dx&\leq\int_BS(f_1\chi_{2B},...,f_m\chi_{2B})(x)|dx\\
&\hspace{-3.5cm}+\sum_{\vec F\in\Lambda}\int_B\(\int_0^R(|\Theta_t(f_1\chi_{F_1},...,f_m\chi_{F_m})(x)|+|\Theta_t(f_1\chi_{F_1},...,f_m\chi_{F_m})(x_B)|)^2\frac{dt}{t}\)^\frac{1}{2}dx\\
&\hspace{-3cm}+\sum_{\vec F\in\Lambda}\int_B\(\int_R^\infty|\Theta_t(f_1\chi_{F_1},...,f_m\chi_{F_m})(x)-\Theta_t(f_1\chi_{F_1},...,f_m\chi_{F_m})(x_B)|^2\frac{dt}{t}\)^\frac{1}{2}dx\\
&=I+II+III
\end{align*}
where
$$\Lambda=\{(F_1,...,F_m):F_i=2B\text{ or }F_i=(2B)^c\}\backslash\{(2B,...,2B)\}.$$
That is $\Lambda$ is the collection of $m$ vectors of sets with with all combinations of components $2B$ and $(2B)^c$ except for $(2B,...,2B)$.  Note that $|\Lambda|=2^m-1$.  We can easily estimate $I$ using that $S$ is bounded from $L^{p_1}\times\cdots\times L^{p_m}$ into $L^p$
\begin{align*}
I\leq|2B|^\frac{1}{p'}||S(f_1\chi_{2B},...,f_m\chi_{2B})||_{L^p}&\less|B|^\frac{1}{p'}\prod_{i=1}^m||f_i\chi_{2B}||_{L^{p_i}}\less|B|\prod_{i=1}^m||f_i||_{L^\infty}.
\end{align*}
Then to bound $II$, take $\vec F\in\Lambda$, $x\in B$, and we first look at the integrand for $x\in B$
\begin{align*}
|\Theta_t(f_1\chi_{F_1},...,f_m\chi_{F_m})(x)|&\less\int t^{-mn}\prod_{i=1}^m\frac{f_i(y_i)\chi_{F_i}(y_i)}{(1+t^{-1}|x-y_i|)^{N+\gamma}}dy_i\\
&\hspace{-3cm}\leq \prod_{j=1}^m||f_j||_{L^\infty} \(\prod_{i:F_i=2B} \int\frac{1}{(1+|x-y_i|)^{N+\gamma}}dy_i\)\(\prod_{i:F_i=(2B)^c}\int_{|y_i|>R}\frac{2^{N+\gamma}\,t^{N+\gamma-n}}{|y_i|^{N+\gamma}}dy_i\)
\end{align*}
\begin{align*}
&\hspace{-4.5cm}\less \prod_{j=1}^m||f_j||_{L^\infty}\(\prod_{i:F_i=(2B)^c}\frac{t^{N+\gamma-n}}{R^{N+\gamma-n}}\)\\
&\hspace{-4.5cm}\less t^{k_0(N+\gamma-n)}R^{-k_0(N+\gamma-n)} \prod_{j=1}^m||f_j||_{L^\infty}
\end{align*}
where $k_0\in\N$ is the number of terms in $\vec F$ such that $F_i=(2B)^c$.  It is important here that $k_0\geq1$.  Now recall that $|\Lambda|=2^m-1$, and it is now trivial to bound $I$,
\begin{align*}
II&\less \prod_{j=1}^m||f_j||_{L^\infty}\int_B\(\int_0^R (t^{k_0(N+\gamma-n)}R^{-k_0(N+\gamma-n)})^2 \frac{dt}{t}\)^\frac{1}{2}dx\less |B|\prod_{j=1}^m||f_j||_{L^\infty}.
\end{align*}
To bound $III$, for a fixed $\vec F\in\Lambda$ and $x\in B$, we look at the integrand
\begin{align*}
&|\Theta_t(f_1\chi_{F_1},...,f_m\chi_{F_m})(x)-\Theta_t(f_1\chi_{F_1},...,f_m\chi_{F_m})(x_B)|\\
&\hspace{3.5cm}\less\int t^{-mn}(t^{-1}|x-x_B|)^\gamma\prod_{i=1}^m\frac{f_i(y_i)\chi_{F_i}(y_i)}{(1+t^{-1}|x-y_i|)^{N+\gamma}}dy_i\\
&\hspace{3.5cm}\less t^{-\gamma}R^\gamma\prod_{i=1}^m||f_i||_{L^\infty}\int\frac{t^{-n}}{(1+t^{-1}|x-y_i|)^{N+\gamma}}dy_i\\
&\hspace{3.5cm}\less t^{-\gamma}R^\gamma\prod_{i=1}^m||f_i||_{L^\infty}.
\end{align*}
Then once more using that $|\Lambda|=2^{m}-1$, we can bound $III$
\begin{align*}
III&\less|B|\prod_{j=1}^m||f_j||_{L^\infty}\(\int_R^\infty (t^{-\gamma}R^{\gamma})^2\frac{dt}{t}\)^\frac{1}{2}\less|B|\prod_{j=1}^m||f_j||_{L^\infty}.
\end{align*}
Then for $f_i\in L^\infty_c$, $i=1,...,m$, \eqref{BMO} holds with constant independent of $f_1,...,f_m$.
\end{proof}

\begin{corollary}\label{c:interpolation}
If $\theta_t$ satisfies \eqref{size}-\eqref{regx} and \eqref{sqbound} holds for $p=2$ and $1<p_1,...,p_m<\infty$, then \eqref{sqbound} holds for all $2\leq p<\infty$ and $1<p_1,...,p_m<\infty$.
\end{corollary}
\begin{proof}
Define the sharp maximal function
\begin{align*}
M^\#f(x)=\sup_{Q\ni x}\frac{1}{|Q|}\int_Q|f(y)-f_Q|dy.
\end{align*}
By definition we have that $||f||_{BMO}=||M^\#f||_{L^\infty}$.  Also it is easy to see that $||M^\#f||_{L^p}\less||Mf||_{L^p}$, where $M$ is the Hardy-Littlewood maximal operator.  Then using the $L^2$ bound of $M$ and the hypothesis on $S$, it follows that for all $f_1,...,f_m\in L^\infty_c$
\begin{align*}
&||M^\#S(f_1,...,f_m)||_{L^2}\less||MS(f_1,...,f_m)||_{L^2}\less\prod_{i=1}^m||f_i||_{L^{p_i}}
\end{align*}
and by assumption by theorem \ref{t:BMO}
\begin{align*}
&||M^\#S(f_1,...,f_m)||_{L^\infty}=||S(f_1,...,f_m)||_{BMO}\less\prod_{i=1}^m||f_i||_{L^\infty}.
\end{align*}
Then by multilinear Marcinkiewicz interpolation, it follows that
\begin{align*}
||M^\#S||_{L^p}\less\prod_{i=1}^m||f_i||_{L^{p_i}}
\end{align*}
for all $f_i\in L^\infty_c$ where $2\leq p<\infty$ and $1<p_1,...,p_m<\infty$ satisfying \eqref{Holder} with constant independent of $f_1,...,f_m$.  Since $L^\infty_c$ is dense in $L^q$ for all $1\leq q<\infty$, it follows that $M^\#S$ is bounded from $L^{p_1}\times\cdots\times L^{p_m}$ into $L^p$ for all $2\leq p<\infty$ and $1<p_1,...,p_m<\infty$.  We have also from a result of Fefferman-Stein \cite{FS2} that $||f||_{L^q}\less||M^\#f||_{L^q}$ when $1\leq q<\infty$ and $f$ satisfies $M^df\in L^q$ where $M^d$ is the dyadic maximal function (in particular when $f\in L^q$ for $1<q<\infty$).  Therefore
\begin{align*}
||S(f_1,...,f_m)||_{L^p}\less||M^\#S(f_1,...,f_m)||_{L^p}\less\prod_{i=1}^m||f_i||_{L^{p_i}},
\end{align*}
which completes the proof.
\end{proof}

\section{Proof of the Theorem \ref{t(1,1)}}

The way we will prove theorem \ref{t(1,1)} is to first assume that $T$ satisfies
\begin{align}
T^{*1}(1,...,1)=\cdots=T^{*m}(1,...,1)=0\label{reduced}
\end{align}
in place of \eqref{Tstarcancel}, and prove that $T$ is bounded.  Then we proceed by using a multilinear version of the T1 paraproduct used in the original T1 theorem by David-Journ\'e \cite{DJ}.  The bilinear version of this paraproduct was constructed in \cite{Hart2}.
\begin{lemma}\label{l:paraproduct}
Given $\beta\in BMO$, there exists a multilinear Calder\'on-Zygmund operator $L$ bounded from $L^{p_1}\times\cdots\times L^{p_m}$ into $L^p$ for all $1<p_i<\infty$ satisfying \eqref{Holder} such that
\begin{align}
L(1,...,1)=\beta\text{ and }L^{*i}(1,...,1)=0\text{ for }i=1,...,m.\label{L(1)}
\end{align}
\end{lemma}
We will give a proof of this lemma at the end of this section.  Now we prove the theorem \ref{t(1,1)} assuming lemma \ref{l:paraproduct}.
\begin{proof}
Denote by $P_t$ be a smooth approximation to identity operators with smooth compactly supported kernels that satisfy
$$
    f=\lim_{t\rightarrow0}P_tf\;\;\;\;\text{ and }\;\;\;\;0=\lim_{t\rightarrow\infty}P_tf
$$
in $\S$ for $f\in\S_0$.  There exist Littlewood-Paley-Stein projection operators $Q_t^{(i)}$ for $i=1,2$ with smooth compactly supported kernels such that $t\frac{d}{dt}P_t^2=Q_t^{(2)}Q_t^{(1)}$.  Using these operators, we decompose $T$ for $f_i\in\S_0$, $i=0,...,m$
\begin{align*}
|\<T(f_1,...,f_m),f_0\>|&=\left|\int_0^\infty t\frac{dt}{t}\<T(P_t^2f_1,...,P_t^2f_m),P_t^2f_0\>\frac{dt}{t}\right|\\
&\hspace{-2.5cm}\leq\sum_{i=0}^m\int_0^\infty\left|\<\Theta_t^{(i)}(f_1,...,f_{i-1},f_0,f_{i+1},...f_m),Q_t^{(1)}f_i\>\right|\frac{dt}{t}\\
&\hspace{-2.5cm}\leq\sum_{i=0}^m\left|\left|\(\int_0^\infty|\Theta_t^{(i)}(f_1,...,f_{i-1},f_0,f_{i+1},...f_m)|^2\frac{dt}{t}\)^\frac{1}{2}\right|\right|_{L^{p_i'}}\left|\left|\(\int_0^\infty|Q_t^{(1)}f_i|^2\frac{dt}{t}\)^\frac{1}{2}\right|\right|_{L^{p_i}}
\end{align*}
where we define $p_0=p'$ and
\begin{align*}
\Theta_t^{(i)}(f_1,...,f_m)=Q_t^{(2)\,*}T^{*i}(P_t^2f_1,...,P_t^2f_m)
\end{align*}
and $T^{*i}$ is the $i^{th}$ formal transpose of $T$ defined by the pairing for $f_0,...,f_m\in\S$
$$\<T^{*i}(f_1,...,f_m),f_0\>=\<T(f_1,...,f_{i-1},f_0,f_{i+1},...,f_m),f_i\>.$$
This type of decomposition was originally done by Coifman-Meyer in \cite{CM}, and then in the bilinear setting in \cite{Hart2}  Since $1<p_1,...,p_m<\infty$, the second term in above can be bounded by $||f_i||_{L^{p_i}}$ using a Littlewood-Paley-Stein estimate for $Q_t^{(1)}$.  We have also assume that $T\in WBP$ which we define now
\begin{definition}
For $M\in\N$, a function $\phi\in C_0^\infty(\R^n)$ is a normalized bump of order $M$ if $\supp(\phi)\subset B(0,1)$ and for all multi-indices $\alpha\in\N_0^n$ with $|\alpha|\leq M$,
\begin{align*}
||\partial^\alpha\phi||_{L^\infty}\leq1.
\end{align*}
An $m$-linear operator $T:\S\times\cdots\times\S^m\rightarrow\S'$ satisfies the weak boundedness property, written $T\in WBP$, if there exists $M\in\N$ such that for all normalized bumps $\phi_0,...,\phi_m\in C_0^\infty$ of order $M$
\begin{align*}
\left|\<T(\phi_1^{x,R},...,\phi_m^{x,R}),\phi_0^{x,R}\>\right|\less R^n
\end{align*}
where $\phi^{x,R}(y)=\phi\(\frac{y-x}{R}\)$.
\end{definition}
It follows that $\theta_t^{(i)}$ satisfy \eqref{size}-\eqref{regx} for $i=0,1,...,m$ when $|x-y|\less t$ since $T\in WBP$ and for $|x-y|\more t$ using the kernel representation of $T$ (for details see \cite{Hart2}).  It follows from Theorem \ref{t:main} and \eqref{Tcancel} that \eqref{Lpbound} holds for all $2\leq p<\infty$ and $1<p_i<\infty$ where $S$ is the square function associated to $\Theta_t^{(0)}$ defined by \eqref{sqfunction}.  Also it follows from \cite{Hart1} or \cite{GLMY} that \eqref{Lpbound} holds for all $1<p,p_i<\infty$ where $S$ is the square function associated to $\Theta_t^{(i)}$ defined by \eqref{sqfunction} for $i=1,...,m$.  Now fix $2\leq p<\infty$ and $1<p_i<\infty$ such that \eqref{Holder} holds.  For example take $p_i=2m$ and $p=2$.  Then $p_i'=\frac{2m}{2m-1}>1$ for $i=1,...,m
 $.  Using this choice of indices, it follows from \eqref{sqbound} that $T$ is bounded from $L^{2m}\times\cdots\times L^{2m}$ into $L^2$, and hence is bounded from $L^{p_1}\times\cdots\times L^{p_m}$ into $L^p$ for all $1<p_1,...,p_m<\infty$ such that \eqref{Holder} holds (see for example \cite{GT1}).  Here we have used that
\begin{align*}
\left|\left|\(\int_0^\infty|\Theta_t^{(0)}(f_1,...,f_m)|^2\frac{dt}{t}\)^\frac{1}{2}\right|\right|_{L^p}\less\prod_{i=1}^m||f_i||_{L^{p_i}}
\end{align*}
and that for $j=1,...,m$
\begin{align*}
\left|\left|\(\int_0^\infty|\Theta_t^{(j)}(f_1,...,f_{j-1},f_0,f_{j+1},...,f_m)|^2\frac{dt}{t}\)^\frac{1}{2}\right|\right|_{L^{p_i'}}\less||f_0||_{L^{p'}}\prod_{i\neq j}||f_i||_{L^{p_i}}.
\end{align*}
This proves the reduces case of theorem \ref{t(1,1)} where we assumed \eqref{reduced} in place of \eqref{Tstarcancel}.  Now assuming that lemma \ref{l:paraproduct} holds, we prove the full theorem \ref{t(1,1)} where $T$ satisfies \eqref{Tstarcancel}.  Given $T$ satisfying the hypotheses of theorem \ref{t(1,1)}, by lemma \ref{l:paraproduct} there exist operators bounded $m$-linear Clader\'on-Zygmund operators $L_1,...,L_m$ such that
\begin{align*}
L_i^{*i}(1,...,1)=T^{*i}(1,...,1)\text{ and }L_i^{*j}(1,...,1)=0\text{ for }i\neq j.
\end{align*}
Define
\begin{align*}
\widetilde T(f_1,...,f_m)=T(f_1,...,f_m)-\sum_{i=1}^mL_i(f_1,...,f_m).
\end{align*}
Then $\widetilde T$ satisfies for $i=1,...,m$
\begin{align*}
\widetilde T^{*i}(1,...,1)&=T^{*i}(1,...,1)-\sum_{i=1}^mL_i^{*i}(1,...,m)=0.
\end{align*}
Now for any dyadic cube $Q\subset\R^n$ we bound $\widetilde T$ as in \eqref{Tcancel}
\begin{align*}
\int_Q\(\int_0^{\ell(Q)}|Q_t\widetilde T(P_tb_Q^1,...,P_tb_Q^m)|^2\frac{dt}{t}\)^\frac{1}{2}dx&\leq\sum_{i=1}^m\int_Q\(\int_0^{\ell(Q)}|Q_tL_i^{*i}(P_tb_Q^1,...,P_tb_Q^m)|^2\frac{dt}{t}\)^\frac{1}{2}dx\\
&\hspace{.75cm}+\int_Q\(\int_0^{\ell(Q)}|Q_tT(P_tb_Q^1,...,P_tb_Q^m)|^2\frac{dt}{t}\)^\frac{1}{2}dx.
\end{align*}
The second term is bounded by $|Q|$ by hypothesis.  If we prove that the square function associate to each term $Q_tL_i^{*i}(P_tf_1,...,P_tf_m)$ is bounded from $L^{q_1}\times\cdots\times L^{q_m}$ into $L^q$, then we bound the first term as well and we can apply the reduced version to complete the proof.  So we have reduced the proof to showing that \eqref{Lpbound} holds for $2\leq p<\infty$ and $1<p_1,...,p_m<\infty$ for $\Theta_t(f_1,...,f_m)=Q_tL_i^{*i}(P_tf_1,...,P_tf_m)$ with its associated kernel $\theta_t(x,y_1,...,y_m)$ and square function $S$ as in \eqref{sqfunction}.  Since $L_i$ is bounded, it follows that
\begin{align*}
|\theta_t(x,y_1,...,y_m)|&=|\<L_i(\varphi_t^{y_1},...,\varphi_t^{y_m}),\psi_t^x\>|\less||\psi_t||_{L^2}\prod_{i=1}^m||\varphi_t||_{L^{2m}}\less t^{-mn}.
\end{align*}
Also, if $|x-y_{i_0}|>4t$ it follows that
\begin{align}
|\theta_t(x,y_1,...,y_m)|&=\left|\int\ell(u,v_1,...,v_m)\psi_t(x-u)\prod_{i=1}^m\varphi_t(y_i-v_i)du\,dv\right|\notag\\
&=\left|\int(\ell(u,v_1,...,v_m)-\ell(x,v_1,...,v_m))\psi_t(x-u)\prod_{i=1}^m\varphi_t(y_i-v_i)dv\right|\notag\\
&\less\int\frac{|x-u|^\gamma}{\(\sum_{i=1}^m|x-v_i|\)^{mn+\gamma}}|\psi_t(x-u)|\prod_{i=1}^m|\varphi_t(y_i-v_i)|du\,dv\notag\\
&\less\int_{|v_{i_0}-y_{i_0}|<t}\int_{|x-u|<t}\frac{t^\gamma}{|x-v_{i_0}|^{mn+\gamma}}t^{-(m+1)n}du\,dv\notag\\
&\less\frac{t^{-mn}}{(1+t^{-1}|x-y_{i_0}|)^{mn+\gamma}}.\label{bound}
\end{align}
In this computation we use that $|x-y_{i_0}|>4t$ to replace $|x-v_{i_0}|$ with $|x-y_{i_0}|+t$:  For $v_{i_0}$ such that $|v_{i_0}-y_{i_0}|<t$, we have
\begin{align*}
|x-v_{i_0}|\geq|x-y_{i_0}|-|y_{i_0}-v_{i_0}|>\frac{1}{2}|x-y_{i_0}|+t.
\end{align*}
Since $|\theta_t(x,y_1,...,y_m)|\less t^{-mn}$ as well, it follows that $\theta_t$ satisfies \eqref{bound} for all $x,y_{i_0}\in\R^n$ and $i_0=1,...,m$ (not just for $|x-y_{i_0}|>4t$).  Then it follows that $\theta_t$ satisfies \eqref{size}
\begin{align*}
|\theta_t(x,y_1,...,y_m)|&\less\prod_{i=1}^m\(\frac{t^{-mn}}{(1+t^{-1}|x-y_{i_0}|)^{mn+\gamma}}\)^{1/m}\\
&\less\prod_{i=1}^m\frac{t^{-n}}{(1+t^{-1}|x-y_{i_0}|)^{n+\gamma/m}}.
\end{align*}
It follows as well that $\theta_t$ satisfies \eqref{regy} and \eqref{regx}.  Consider
\begin{align*}
|\theta_t(x,y_1,...,y_m)-\theta_t(x',y_1,...,y_m)|&=\left|\<L_i(\varphi_t^{y_1},...,\varphi_t^{y_m}),\psi_t^x-\psi_t^{x'}\>\right|\\
&\less t^{-mn}(t^{-1}|x-x'|).
\end{align*}
When coupled with the size condition \eqref{size}, this estimate is sufficient for \eqref{regx} if we allows for a possibly smaller regularity parameter $\gamma$.  Then by symmetric arguments for $y_1,...,y_m\in\R^n$, $\theta_t$ satisfies \eqref{size}-\eqref{regx}.  Moreover, since $L_i$ is bounded it follows that $L_i(1,...,1)\in BMO$ and so
\begin{align*}
|\Theta_t(1,...,1)(x)|^2dx\frac{dt}{t}&=|Q_tL_i(1,...,1)|^2dx\frac{dt}{t}
\end{align*}
is a Carleson measure.  Therefore by proposition \ref{p:Carlesonbound} and corollary \ref{c:interpolation}, it follows that \eqref{Lpbound} holds for the square function $S$ associated to $\Theta_t=Q_tL_i(P_t\otimes\cdots\otimes P_t)$ for any $2\leq p<\infty$, $1<p_1,...,p_m<\infty$ satisfying \eqref{Holder} and for each $i=1,...,m$.  Therefore the second term above can be bounded since $q\geq2$
\begin{align*}
\int_Q\(\int_0^\infty|Q_tT(P_tb_Q^1,...,P_tb_Q^m)|^2\frac{dt}{t}\)^\frac{1}{2}dx&\leq|Q|^\frac{q}{q'}\left|\left|\(\int_0^\infty|\Theta_t(b_Q^1,...,b_Q^m)|^2\frac{dt}{t}\)^\frac{1}{2}\right|\right|_{L^q}^q\\
&\less|Q|^\frac{q}{q'}\prod_{i=1}^m||b_Q^1||_{L^{q_i}}^q\\
&\less|Q|.
\end{align*}
Therefore $\widetilde T$ satisfies \eqref{Tcancel} as well and hence is bounded for from $L^{p_1}\times\cdots\times L^{p_m}$ into $L^p$ for all $1<p_1,...,p_m<\infty$ satisfying \eqref{Holder}.  It follows easily that $T$ is bounded on the same spaces since $\widetilde T$ and $L_i$ for each $i=1,...,m$ are.
\end{proof}

Finally we prove the paraproduct construction in lemma \ref{l:paraproduct}.

\begin{proof}
Let $P_t$ be a smooth approximation to the identity with convolution kernel supported in $B(0,1)$.  Also fix $\psi\in C_0^\infty$ radial, real-valued with mean zero such that
\begin{align*}
\int_0^\infty\widehat\psi(te_1)^3\frac{dt}{t}=1
\end{align*}
where $e_1=(1,0,...,0)\in\R^n$, and define $Q_tf=\psi_t*f$.  It follows that
$$\int_0^\infty Q_t^3f\frac{dt}{t}=f$$
in $L^p$ for all $1<p<\infty$ and in $H^1$, where $Q_t^3$ is the composition of $Q_t$ with itself three times.  Now define $L$ with kernel $\ell(x,y_1,...,y_m)$ by the following
\begin{align*}
L(f_1,...,f_m)&=\int_0^\infty L_t(f_1,...,f_m)\frac{dt}{t}=\int_0^\infty Q_t\((Q_t^2\beta)\prod_{i=1}^mP_tf_i\)\frac{dt}{t}\\
\ell(x,y_1,...,y_m)&=\int_0^\infty\ell_t(x,y_1,...,y_m)\frac{dt}{t}=\int_0^\infty\int\psi_t(x-u)Q_t^2\beta(u)\prod_{i=1}^m\varphi_t(u-y_i)du\frac{dt}{t}.
\end{align*}
We start by analyzing $L_t$.  Define the non-negative measure $d\mu$ on $\R^{n+1}_+$ by
\begin{align*}
&d\mu(x,t)=|\widetilde L_t(1,...,1)(x)|^2dx\frac{dt}{t}=|Q_t^2\beta(x)|^2dx\frac{dt}{t}\\
&\text{ where }\;\;\;\widetilde L_t=M_{Q_t^2\beta}\prod_{i=1}^mP_tf_i.
\end{align*}
It follows then that $d\mu(x,t)$ is a Carleson measure.  It is straightforward to show that the kernels of $\widetilde L_t$ satisfy \eqref{size}-\eqref{regx} as well (in fact we can take $N>mn+1$ since $\varphi,\psi\in C_0^\infty$, which we will use later).  The smoothness in $x$ is easy to show since we have that $\widetilde L_t$ is multiplied by 
\begin{align*}
Q_t^2\beta(x)=\int\psi_t(x-u)Q_t\beta(u)du
\end{align*}
and $\psi_t$ is smooth.  So by proposition \ref{p:Carlesonbound} and corollary \ref{c:interpolation}, we have
\begin{align*}
\left|\left|\(\int_0^\infty|\widetilde L_t(f_1,...,f_m)|^2\frac{dt}{t}\)^\frac{1}{2}\right|\right|_{L^p}\less\prod_{i=1}^m||f_i||_{L^{p_i}}
\end{align*}
for all $2\leq p<\infty$ and $1<p_1,...,p_m<\infty$.  Then for any $f_0,f_1,...,f_m\in\S$ with $||f_0||_{L^{p'}}\leq1$
\begin{align*}
|\<L(f_1,...,f_m),f_0\>|&\leq\int_0^\infty\left|\int Q_t^2\beta(x)\prod_{i=1}^mP_tf_i(x)Q_tf_0(x)dx\right|\frac{dt}{t}\\
&\leq\left|\left|\(\int_0^\infty|\widetilde L_t(f_1,...,f_m)|^2\frac{dt}{t}\)^\frac{1}{2}\right|\right|_{L^p}\left|\left|\(\int_0^\infty| Q_tf_0|^2\frac{dt}{t}\)^\frac{1}{2}\right|\right|_{L^{p'}}\\
&\less||f_0||_{L^{p'}}\prod_{i=1}^m||f_i||_{L^{p_i}}.
\end{align*}
Therefore $L$ is bounded for appropriate indices $p,p_1,...,p_m$.  It also follows that $\ell$ is a Calder\'on-Zygmund kernel.  To see this, take $d=\sum_{i=1}^m|x-y_i|$ and use \eqref{size} to compute
\begin{align*}
|\ell(x,y_1,...,y_m)|&\less d^{-(N+\gamma)}\int_0^dt^{N+\gamma-mn}\frac{dt}{t}+\int_d^\infty t^{-mn}\frac{dt}{t}\less d^{-mn}.
\end{align*}
Similarly we have
\begin{align*}
|\ell(x,y,z)-\ell(x',y,z)|&\less|x-x'|^\gamma d^{-(N+\gamma)}\int_0^d t^{N+\gamma-mn}\frac{dt}{t}+|x-x'|^\gamma\int_d^\infty t^{-mn-\gamma}\frac{dt}{t}\\
&\less |x-x'|^\gamma d^{-(mn+\gamma)}
\end{align*}
With symmetric arguments for the regularity in $y_1,...,y_m$, it follows that the kernel $\ell$ is an $m$-linear Calder\'on-Zygmund kernel.  So $L$ is an $m$-linear Calder\'on-Zygmund operator, and is bounded from $L^{p_1}\times\cdots\times L^{p_m}$ into $L^p$ for all $1<p_i<\infty$ when \eqref{Holder} holds.

Now we show \eqref{L(1)}.  Let $\eta\in C_0^\infty$ with $\eta\equiv1$ on $B(0,1)$, $\supp(\eta)\subset B(0,2)$, and $\eta_R(x)=\eta(x/R)$.  Let $\phi\in C_0^\infty$ with mean zero and $N$ such that $\supp(\phi)\subset B(0,N)$.  Then to compute $L(1,...,1)$
\begin{align}
\<L(1,...,1),\phi\>&=\lim_{R\rightarrow\infty}\int_{R/4}^\infty\int Q_t\phi(x)\[P_t\eta_R(x)\]^m Q_t^2\beta(x)dx\frac{dt}{t}\notag\\
&\hspace{2cm}+\lim_{R\rightarrow\infty}\int_0^{R/4}\int Q_t\phi(x)\[P_t\eta_R(x)\]^m Q_t^2\beta(x)dx\frac{dt}{t}.\label{L(1,1)}
\end{align}
We may write this only if the two limits on the right hand side of the equation exist.  As we are taking $R\rightarrow\infty$ and $N$ is a fixed quantity determined by $\phi$, without loss of generality assume that $R>2N$.  Note that for $t\leq R/4$ and $|x|<N+t$,
$$\supp(\varphi_t(x-\cdot))\subset B(x,t)\subset B(0,N+2t)\subset B(0,R).$$
Since $\eta_R\equiv1$ on $B(0,R)$, it follows that $P_t\eta_R(x)=1$ for all $|x|<N+t$ when $t\leq R/4$.  Therefore
$$\lim_{R\rightarrow\infty}\int_{R/4}^\infty\int Q_t\phi(x)\[P_t\eta_R(x)\]^mQ_t^2\beta(x)dx\frac{dt}{t}=\int\int_0^\infty Q_t^3\phi(x)\frac{dt}{t}\,\beta(x)dx=\<\beta,\phi\>,$$
where we have used that Calder\'on's reproducing formula holds in $H^1$.  This fact is due originally due to Folland-Stein \cite{FS} in the discrete setting and by Wilson in \cite{W} in the continuous setting as used here.  For any $t>0$
\begin{align}
&||P_t\eta_R||_{L^1}\less||\varphi_t|||_{L^1}||\eta_R||_{L^1}\less R^n,\label{estimate1}\\
&||P_t\eta_R||_{L^\infty}\leq||\varphi_t||_{L^1}||\eta_R||_{L^\infty}=1,\label{estimate2}
\end{align}
and for any $x\in\R^n$
\begin{align}
|Q_t\phi(x)|=\left|\int(\psi_t(x-y)-\psi_t(x))\phi(y)dy\right|\less\int t^{-n}(t^{-1}|y|)|\phi(y)|dy \less t^{-(n+1)}.\label{estimate3}
\end{align}
Therefore
\begin{align}
\int_{R/4}^\infty\int|Q_t\phi(x)\[P_t\eta_R(x)\]^m Q_t^2\beta(x)|dx\frac{dt}{t}&\notag\\
&\hspace{-2.5cm}\leq\int_{R/4}^\infty||P_t\eta_R||_{L^1}||P_t\eta_R||_{L^\infty}^{m-1}||Q_t^2\beta||_{L^\infty}||Q_t\phi||_{L^\infty}\frac{dt}{t}\notag\\
&\hspace{-2.5cm}\less R^n\int_{R/4}^\infty t^{-(n+1)}\frac{dt}{t}\less R^{-1}.\label{secondterm}
\end{align}
Hence the second limit in \eqref{L(1,1)} exists and tends to $0$ as $R\rightarrow\infty$.  Then $\<L(1,...,1),\phi\>=\<\beta,\phi\>$ for all $\phi\in  C_0^\infty$ with mean zero and hence $L(1,...,1)=\beta$ as an element of $BMO$.  Again for any $\phi\in C_0^\infty$ with mean zero and $\supp(\phi)\subset B(0,N)$, we have for $i=1,...,m$
\begin{align}
\<L^{i*}(1,...,1),\phi\>&=\lim_{R\rightarrow\infty}\int_0^{R/4}\int_{|x|<N+t}Q_t^2\beta(x)P_t\phi(x)[P_t\eta_R(x)]^{m-1}Q_t\eta_R(x)dx\frac{dt}{t}\notag\\
&\hspace{1cm}+\lim_{R\rightarrow\infty}\int_{R/4}^\infty\int_{|x|<N+t}Q_t^2\beta(x)P_t\phi(x)[P_t\eta_R(x)]^{m-1}Q_t\eta_R(x)dx\frac{dt}{t}.\label{L*(1,1)}
\end{align}
Once more without loss of generality take $R>2N$.  When $|x|<N+t$ and $t\leq R/4$
$$\supp(\psi_t(x-\cdot))\subset B(x,t)\subset B(0,N+2t)\subset B(0,R)$$
and hence $Q_t\eta_R(x)=Q_t1(x)=0$.  With this it is apparent that the first limit in \eqref{L*(1,1)} is $0$.  Similar to \eqref{estimate1}-\eqref{estimate3}, for the terms of \eqref{L*(1,1)} we have $||P_t\eta_R||_{L^1}\less R^n$, $||Q_t\eta_R||_{L^\infty}\less1$, and $||P_t\phi||_{L^\infty}\less t^{-(n+1)}$.  So the second term of \eqref{L*(1,1)} tends to $0$ as $R\rightarrow\infty$ just like the second term in computing $L(1,...,1)$ from \eqref{secondterm}.  Then $L^{*1}(1,...,1)=0$, which concludes the proof of lemma \ref{L(1)}.
\end{proof}

\bibliographystyle{amsplain}

\end{document}